\documentclass[12pt]{amsart}
\usepackage{amsmath}
\usepackage{amsmath}
\usepackage{mathrsfs}
\usepackage{mathrsfs}
\usepackage{mathrsfs}
\usepackage{mathrsfs}
\usepackage{amssymb}
\usepackage{amsmath}
\usepackage{amsmath}
\usepackage{amsmath}
\usepackage{amsmath}
\usepackage{amsmath}
\usepackage{mathrsfs}
\usepackage{amssymb}
\usepackage{bbm}
\usepackage{amsmath}
\usepackage{mathrsfs}
\usepackage{amsfonts}
\usepackage{amsmath,amssymb,amsfonts}
\usepackage[all]{xy}
\usepackage{hyperref}
\usepackage{a4wide}
\usepackage[mathscr]{eucal}

\theoremstyle{plain}
\newtheorem{thm}{Theorem}[section]
\newtheorem{lem}[thm]{Lemma}
\newtheorem{cor}[thm]{Corollary}

\theoremstyle{definition}

\newtheorem*{construction}{Construction}
\newtheorem{defn}{Definition} 

\newtheorem{rmk}[thm]{Remark}
\newtheorem{rmks}[thm]{Remarks}
\numberwithin{thm}{section}

\newcommand{\Ess}{{\rm Ess}}

\newcommand{\Ext}{{\rm Ext}}

\newcommand{\Spec}{{\rm Spec \,}}

\newcommand{\Ker}{{\rm Ker}}

\newcommand{\A}{{\mathbb A}}

\newcommand{\G}{{\mathbb G}}

\newcommand{\N}{{\mathbb N}}
\renewcommand{\P}{{\mathbb P}}
\newcommand{\Q}{{\mathbb Q}}

\newcommand{\X}{{\mathbb X}}
\newcommand{\Y}{{\mathbb Y}}

\newcommand{\Rep}{\text{\sf Rep}}

\begin{document}
\title{The Homotopy Sequence of Nori's Fundamental Group}
\author{ Lei Zhang }
 \address{
Universit\"at Duisburg-Essen, FB6, Mathematik, 45117 Essen, Germany}
\email{lei.zhang@uni-due.de}
\thanks{This work was supported by the Sonderforschungsbereich/Transregio 45 "Periods, moduli
spaces and the arithmetic of algebraic varieties" of the DFG}
\date{April 15, 2012}

\begin{abstract} In this paper, we investigate the necessary sufficient conditions for
the exactness of the homotopy sequence of Nori's fundamental group
and apply these to various special situations to regain some
classical theorems and give a counter example to show the conditions
are not always satisfied. This work is partially based on the earlier
work of H.Esnault, P.H.Hai , E.Viehweg.
\end{abstract}

\maketitle
\section{Introduction}

If $f:X\to S$ is a separable proper morphism with geometrically
connected fibres between locally noetherian connected schemes, $x\to
X$ is a geometric point with image $s\to S$, Grothendieck shows in
\cite[Expos\'e X, Corollaire 1.4]{Gr} that one has a homotopy exact
sequence for the \'etale fundamental group:
$$\pi_1^{\text{\'et}}(\bar{X_s},x)\to\pi_1^{\text{\'et}}(X,x)\to\pi_1^{\text{\'et}}(S,s)\to1.$$
A similar case is that one can take $X,Y$ to be two locally
noetherian connected $k$-schemes with $k=\bar{k}$ and suppose $Y$ is
proper over $k$, so if $K$ is an algebraically closed field
containing $k$ and if we take a $K$-point $z=(x,y): \Spec(K)\to
X\times_kY$, then we get a canonical morphism of topological groups
$$\pi_1^{\text{\'et}}(X\times_kY,z)\to\pi_1^{\text{\'et}}(X,x)\times\pi_1^{\text{\'et}}(Y,y).$$
Again Grothendieck shows in \cite[Expos\'e X, Corollaire 1.7]{Gr}
that the canonical homomorphism is an isomorphism. This is called
the K\"unneth formula for the \'etale fundamental group. If
$X\times_kY$ admits a $k$-rational point then the K\"unneth formula
is a direct consequence of the homotopy exact sequence.

Let $X$ be a reduced connected locally noetherian scheme over a
field $k$, $x\in X(k)$ be a rational point. Let $N(X,x)$ be
the category whose objects consist of triples $(P,G,p)$ (where $P$
is an FPQC $G$-torsor over $X$, $G$ is a finite group scheme, $p\in
P(k)$ is a $k$-rational point lying over $x$), whose morphisms are
morphisms of $X$-schemes which intertwine the group actions and  preserve the points.
M.Nori proved in \cite[Part I, Chapter II, Proposition 2]{Nori} that
the projective limit $\varprojlim_{N(X,x)}G$ exists in the category
of $k$-group schemes (in the projective system we associate to each
index $(P,G,p)$ the group $G$). Then he defined the fundamental
group $\pi^N(X,x)$ to be the projective limit
$\varprojlim_{N(X,x)}G$. If $X$ is in addition proper over $k$ and if $k$ is
perfect, Nori gave in \cite[Part I, Chapter I]{Nori} a Tannakian
description of his fundamental group: he defined $\pi^N(X,x)$ to be
the Tannakian group of the neutral Tannakian category of $\Ess(X)$
(the essentially finite vector bundles on $X$) with the fibre
functor $x^*: V\mapsto V|_x$, and he showed that this definition is
the same as the one defined by the projective limit.

The main purpose of this article is to study the analogues of the
homotopy sequence and K\"unneth formula for Nori's fundamental
group. Since in Nori's fundamental group we only deal with rational
points, if the homotopy sequence is exact then K\"unneth formula
always automatically follows.

In \cite{EHV}[Section 2] H.Esnault, P.H.Hai , E.Viehweg, give a
counterexample which shows that homotopy sequence of Nori's
fundamental group is not always exact even for $X\to S$ projective
smooth and $S$ projective smooth as well. And then they give a
necessary and sufficient condition for the exactness of the homotopy
sequence of Nori's fundamental group under the assumption that $S$
is a proper $k$-scheme. But unfortunately there is a gap in the
argument for the necessary and sufficient condition. In this
article, our first goal is to reformulate some similar conditions to
make everything work. These works are contained in Theorem 2.1 and
Theorem 3.1, where we correct the mistake, improve the arguments and
make the wonderful ideas hidden in that article right and clean. The
upshot is that in Theorem 2.1 we don't have to assume $S$ to be
proper, so the result applies to the general definition of Nori's
fundamental group.

Then we make two applications of Theorem 2.1 and Theorem 3.1. We
first apply the criterion to show that the homotopy sequence for the
\'etale quotient of Nori's fundamental group is exact. The argument
is independent of Grothendieck's theory of the \'etale fundamental
group which was developed in \cite[Expos\'e X]{Gr}, so it can be
seen as a new proof of the homotopy exact sequence for the \'etale
fundamental group (in the language of Nori's fundamental group).

In \cite{MS}[Theorem 2.3] V.B.Mehta and S.Subramanian proved that
K\"unneth formula holds for Nori's fundamental group if both $X$ and
$Y$ are proper $k$-schemes. In \S 3, we apply Theorem 3.1 to give a
neat proof for the K\"unneth formula of the local quotient of Nori's
fundamental group. This can be thought of as a new proof of
\cite{MS}[Proposition 2.1] which is the key point for the proof of
\cite{MS}[Theorem 2.3].

In the end of this paper, we give a counterexample to show that
\cite{MS}[Theorem 2.3] does not work if $X$ or $Y$ is not proper.
We prove that if $X=\A_k^1$ and $Y=E$ is a supersingular elliptic
curve and $k$ is an algebraically closed field in positive characteristic, then the K\"unneth formula is always false. But, in contrast, the K\"unneth formula
always holds in this case for the \'etale fundamental group.  This also provides another counter example to show the failure of
the exactness of the homotopy
sequence for Nori's fundamental group.\\\\
{\bf Acknowledgments:}  I would like to express my deepest gratitude
to my advisor H\'el\`ene Esnault for proposing to me the questions
and guiding me through this topic with patience. I thank Vikram
Mehta for some enlightening discussions. I also thank Ph\`ung H\^{o}
Hai, Nguy\^e\~n Duy T\^an, Jilong Tong, Andre Chatzistamatiou, Kay
R\"ulling for a lot of helpful discussions and encouragements.

\maketitle
\section{The General Criterion}

\begin{defn}
Let $X$ be a reduced connected locally noetherian scheme over a  field $k$, $x\in X(k)$
be a rational point. We call a triple $(P,G,p)\in N(X,x)$ a
$G$-saturated torsor if the canonical map $\pi^N(X,x)\to G$ is
surjective.
\end{defn}


\begin{defn}Let $f:X\to S$ be a map of schemes, $\mathscr{F}$ be a sheaf of $O_X$-modules, $s:\Spec(\kappa(s))\hookrightarrow S$ a point, then we get a Cartesian diagram:
 $$\xymatrix{X_s\ar[r]^-{t}\ar[d]^-{g}&X\ar[d]^-f\\\Spec(\kappa(s))\ar[r]^-s&S}.$$We say $\mathscr{F}$
satisfies base change at $s$ if the canonical map
$$s^*f_*\mathscr{F}\to g_*t^*\mathscr{F}$$ is surjective. Note that
if $f$ is proper, $S$ is locally noetherian, $\mathscr{F}$ is
coherent and flat over $S$ then $\mathscr{F}$ satisfies base change
at $s$ if and only if the above canonical map is an isomorphism (see
\cite{Hart}{\rm[Chapter III, Theorem 12.11]}).
\end{defn}

\begin{defn}We call a morphism of schemes $f:X\to S$ separable if it is flat and if for $\forall s\in S$ the fibre $X_{{s}}$ is geometrically reduced over $\kappa(s)$.\cite{Gr}[Expos\'e X, D\'efinition 1.1]
\end{defn}

\begin{thm}\footnote{In the original article  \cite{EHV} the statement was that the exactness is equivalent to base change. But it seems that only base change is not enough to deduce the exactness, so we add extra conditions to make the argument. In addition, we removed the properness assumption on $S$ in the original statement.} {\rm (H.Esnault, P.H.Hai , E.Viehweg)} Let $f:X\to S$ be a separable proper
morphism with geometrically connected fibres between two reduced
 connected locally noetherian schemes over a perfect
field $k$. We suppose further that $S$ is irreducible. Let $x\in
X(k)$, $s\in S(k)$ and assume $f(x)=s$. Then the following
conditions are equivalent: \begin{enumerate}\item the sequence
$$\pi^N(X_s,x)\to\pi^N(X,x)\to\pi^N(S,s)\to1$$ is exact;\item for
any $G$-saturated torsor $(P,G,p)$ with structure map $\pi: P\to X$,
$\pi_*O_P$ satisfies base change at $s$  and the image of the
composition $\pi^N(X_s,x)\to \pi^N(X,x)\to G$ is a normal subgroup
of $G$; \item for any $G$-saturated torsor $(P,G,p)$ with structure
map $\pi: P\to X$, $\pi_*O_P$ satisfies base change at $s$ and there
is a $G'$-saturated torsor $\pi': P'\to S$ together with a morphism
$(P,G)\xrightarrow{\theta} (P',G')$ satisfy that the
$\theta$-induced map $(\pi'_*O_{P'})_s\to (f_*\pi_*O_P)_s$  is an
isomorphism.
\end{enumerate}
\end{thm}
\begin{proof}"$(1)\Longrightarrow(2)$"
If the homotopy sequence is exact then clearly the image of
$\pi^N(X_s,x)\to \pi^N(X,x)\to G$ (which is denoted by $H$) is normal in $G$. The exactness
also gives us a commutative diagram
$$\xymatrix{\pi^N(X,x)\ar[r]\ar[d]&\pi^N(S,s)\ar[d]\\ G\ar[r]& G/H}.$$ This commutative diagram gives us a $G/H$-saturated torsor $(P',G/H,p')$ over $S$ and
a morphism in $N(X,x)$: $$\lambda:(P,G,p)\to
(P'\times_SX,G/H,p'\times_SX)\cong (P/H,G/H,p).$$ Let $W'$ be the
push forward of the structure sheaf of $P'$ to $S$, $V:=\pi_*O_P$,
$W:=f^*W'$. Let $\lambda^*: W\to V$ be the map induced by $\lambda$.
If we pull-back $\lambda^*$ to $X_s$ then we get a morphism in the
category of essentially finite vector bundles because $V|_{X_s}$
(resp.$W|_{X_s}$) is the 0-th direct image of the structure sheaf of
the torsor $P\times_XX_s$ (resp.$P'\times_SX_s$). From
\cite{Nori}[Part I, Chapter I, Proposition 2.9], this $\lambda^*$
corresponds, via Tannakian duality, to the morphism
$$k[G]^{\pi^N(X_s,x)}=k[G]^H=k[G/H]\to k[G]$$ in the category of
$\Rep_k(\pi^N(X_s,x))$. Hence $W|_{X_s}$ is the maximal trivial
subbundle of $V|_{{X_s}}$. But
$H^0(X_s,V|_{X_s})\otimes_kO_{X_s}\subseteq V|_{X_s}$ is the maximal
trivial subbundle (see lemma 2.2 below), thus the canonical map
$$W|_{X_s}=H^0(X_s,W|_{X_s})\otimes_kO_{X_s}\to H^0(X_s,V|_{X_s})\otimes_kO_{X_s}
$$ is an isomorphism. But note that the above map factors
$W|_{X_s}\to f^*f_*V|_{X_s}$. This implies $f^*f_*V|_{X_s}\to
H^0(X_s,V|_{X_s})\otimes_kO_{X_s}$ is an isomorphism, so base change
is satisfied.

"$(2)\Longrightarrow(3)$" Let $H\subseteq G$ be the image of the
composition $\pi^N(X_s,x)\to \pi^N(X,x)\to G$. Since it is normal we
get a $G/H$-torsor $P/H$ on $X$. If $W$ is the push-forward of the
structure sheaf of $P/H$ to $X$ and $V:=\pi_*O_P$, then we know from
our assumption that $W$ and $V$ satisfy base change at $s$. Let
$\lambda: W\hookrightarrow V$ be the imbedding induced from $P\to P/H$,
then we have the following commutative diagram of sheaves on
$X_s$:$$\xymatrix{
f^*f_*W|_{X_s}\ar[r]^-{a_1}\ar[d]^-{f^*f_*\lambda}&H^0(X_s,W|_{X_s})\otimes_kO_{X_s}\ar[d]^-{H^0(X_s,\lambda|_{X_s})}\ar[r]^-{a_2}&W|_{X_s}\ar[d]^-{\lambda|_{X_s}}\\f^*f_*V|_{X_s}\ar[r]^-{a_3}&H^0(X_s,V|_{X_s})\otimes_kO_{X_s}\ar[r]^-{a_4}&V|_{X_s}
}.$$ By base change $a_1, a_3$ are isomorphisms. Since
$\lambda|_{X_s}$ corresponds via Tannakian duality to
$k[G]^H\hookrightarrow k[G]$ (in the category
$\Rep_k(\pi^N(X_s,x))$), $W|_{X_s}$ is imbedded as the maximal
trivial subbundle of $V|_{X_s}$. Hence $a_2$ and
$H^0(X_s,\lambda|_{X_s})$ are isomorphisms. So $f^*f_*\lambda$ is
also an isomorphism. In particular
$$(f_*\lambda)_{s}:(f_*W)_{s}\to(f_*V)_{s}$$ is an isomorphism. Let
$r\in\N$ be the rank of $W$. For any point $t\in S$, by applying lemma 2.2 to $X_t\times_{\kappa(t)}\overline{\kappa(t)}/\overline{\kappa(t)}$ we know that the following map is an imbedding:
$$H^0(X_{t},W|_{X_t})\otimes_{\kappa(t)}O_{X_t}\to W|_{X_t}.$$ So $\dim_{\kappa(t)}(H^0(X_{t},W|_{X_t}))\leq r$.
But on the other hand, since $W$ satisfies base change at $s$,
$r=\dim_{k}(H^0(X_{s},W|_{X_s}))$ reaches the minimal dimension (the
dimension at the generic point), so by semi-continuity theorem we
have
 $$\dim_{\kappa(t)}(H^0(X_{t},W|_{X_t}))\geq \dim_{k}(H^0(X_{s},W|_{X_s}))=r.$$ This implies $H^0(X_{t},W|_{X_t})$ has constant
dimension $r$, and hence $W$ satisfies base change all over $S$. So $f_*W$ a vector bundle. Since $f^*f_*W\to W$ is injective after restricting to all the points of $X$, we have
it is an embedding as a subbundle (i.e. injective and locally split). But since $a_1,a_2$ are isomorphisms, we have $f^*f_*W\to W$ is an isomorphism. Now we can check easily that
$\Spec(f_*W)\to S$ with the canonical $G/H$-action induced from $P/H$ is an FPQC-torsor which satisfies all our conditions in (3).

"$(3)\Longrightarrow(1)$" Let $(P,G,p)$ be any $G$-saturated torsor over $X$, by the assumption $\pi_*O_P$ satisfies base change and there is a $G'$-torsor $P'$ over $S$ with a morphism $\theta: (P,G)\to (P',G')$ which satisfies the conditions in (3). Let $N$ be the image of $\Ker(\pi^N(f))$
in $G$ (where $\pi^N(f)$ is the map $\pi^N(X,x)\to\pi^N(S,s)$), $N'$
be the kernel of $G\to G'$, and $H\subseteq G$ be the image of the
composition $\pi^N(X_s,x)\to \pi^N(X,x)\to G$. We also write
$W:=\pi'_*O_{P'}$ and $V:=\pi_*O_{P}$. We first note that the
$\theta$-induced map $f^*W|_{X_s}\to V|_{X_s}$ corresponds to
$k[G/N']\to k[G]$ in $\Rep_k(\pi^N(X_s,x))$. But from base change of
$V$ and the fact that the $\theta$-induced map $W_{s}\to (f_*V)_{s}$ is
an isomorphism we know that $f^*W|_{X_s}\to V|_{X_s}$ should be the
same as $H^0(X_s,V|_{X_s})\otimes_kO_{X_s}\to V|_{X_s}$ as
subobjects. Thus the canonical imbedding $k[G/N']\hookrightarrow
k[G/H]$ should be an isomorphism. Hence $N'=H$ as subgroups. But
since we have $H\subseteq N\subseteq N'$, so $H=N$ as well. Because
the equality holds for all $G$-saturated torsor $(P,G,p)$, we have
$\pi^N(X_s,x)\to\Ker(\pi^N(f))$ is surjective. This completes the
proof.
\end{proof}

\begin{lem}If $X$ is a reduced
connected proper scheme over a perfect field $k$ with a point
$x\in X(k)$, then for any essentially finite vector bundle
$V$ on $X$ the canonical morphism $\Gamma(X,V)\otimes_{k}
O_X\rightarrow V$ imbeds $\Gamma(X,V)\otimes_k O_X$ as the maximal
trivial subbundle of $V$.
\end{lem}
\begin{proof}Let $\text{\rm Ess}(X)$ be the category of essentially finite vector bunldes, $\omega_x: \text{\rm Ess}(X)\rightarrow
\text{Vec}_k$ be the fibre functor. Then applying $\omega_x$ to the
canonical morphism $\Gamma(X,V)\otimes_k O_X\rightarrow V$ we get
$\text{Hom}_{O_X}(O_X, V)\cong\Gamma(X,V)\rightarrow
V_x\otimes_{O_{X,x}}k=\omega_x(V)$. But note that we have
$\text{Hom}_{O_X}(O_X,
V)\cong\text{Hom}_{\pi^{N}(X,x)}(k,\omega_x(V))$ where $k$ stands
for the dim 1 vector space with trivial $\pi^{N}(X,x)$ action. One
checks readily that under these isomorphisms we get exactly the
canonical injection
$\text{Hom}_{\pi^{N}(X,x)}(k,\omega_x(V))\rightarrow \omega_x(V)$
sending any morphism $k\rightarrow \omega_x(V)$ to the image of
$1\in k$. Since this map imbeds
$\text{Hom}_{\pi^{N}(X,x)}(k,\omega_x(V))$ as the maximal trivial
sub of $\omega_x(V)$. Using Tannakian duality we get our result.
\end{proof}

\subsection{Application to the \'etale quotient}
\begin{defn}Let $X$ be a connected reduced locally noetherian scheme over a perfect field $k$ which admits a rational point $x\in X(k)$. Let $N^{\text{\'et}}(X,x)$
be the full subcategory of $N(X,x)$ whose objects consist of those
$(P,G,p)$ with $G$ finite \'etale. This sub category is filtered so
we can define the \'etale quotient of $\pi^N(X,x)$ to be
$\pi^{\text{\'et}}(X,x):=\varprojlim_{N^{\text{\'et}}(X,x)}G$. We
have an obvious surjection:
$\pi^N(X,x)\twoheadrightarrow\pi^{\text{\'et}}(X,x)$.\end{defn}

\begin{lem}Let $X$ be a geometrically connected reduced locally noetherian scheme over a perfect field $k$ which admits a rational point $x\in X(k)$. Let $(P,G,p)$ be an \'etale torsor over $(X,x)$. This torsor is
$G$-saturated if and only if $P$ is geometrically connected.
\end{lem}
\begin{proof}Since the formation of Nori's fundamental group is
compatible with separable field extensions, we can assume $k=\bar{k}$. \cite{Nori}[Part I, Chapter II, Proposition 5]

"$\Longrightarrow$" Let's take $Q\subseteq P$ to be the connected component of $P$ containing $p$. Now $G$ is an abstract group we can
write the action $\rho:P\times_kG\to P$ as $\coprod_{G} P\to P$ where each component in the direct union is mapped to $P$ via a unique element in $G$. Since
$P$ is an $G$-torsor we have the following cartesian diagram: $$\xymatrix{\coprod_{G} P\ar[r]^-{\rho}\ar[d]^{id^G}&P\ar[d]\\P\ar[r]&X}.$$
If we let $H\subseteq G$ be the maximal subgroup of $G$ which fix $Q$, then we can see by definition that $Q\times_kH\subseteq P\times_kG$ is the intersection of
$\rho^{-1}(Q)$ and ${(id^G)^{-1}(Q)}$. Thus the square $$\xymatrix{\coprod_{H} Q\ar[r]^-{\rho}\ar[d]^{id^H}&Q\ar[d]\\Q\ar[r]&X}$$ is cartesian. Hence $Q$ is
an $H$-torsor. But from the assumption the imbedding $H\to  G$ should be surjective. This tells us $H=G$. But then the map of $G$-torsors $Q\subseteq P$ should also be
an isomorphism. So $P$ is connected.

"$\Longleftarrow$" Let $(P',G',p')\to (P,G,p)$ be any morphism in
$N(X,x)$. Since $P\to X$ is \'etale, we know $P'\to P$ is finite
flat. Thus the image must be both open and closed, and hence it must
be the whole of $P$. But if we pull-back the surjective map $P'\to
P$ via $x\in X(k)$, we will get the group homorphism $G'\to G$. Thus
this homorphism must be surjective. Since $(P',G',p')$ is taken
arbitrarily, it actually shows that $(P,G,p)$ is $G$-saturated.
\end{proof}

\begin{cor}Let $f:X\to S$ be a separable proper
morphism with geometrically connected fibres between two reduced
geometrically connected locally noetherian schemes over a perfect
field $k$. Let $x\in X(k)$, $s\in S(k)$ and assume $f(x)=s$. Then
the homotopy sequence:$$\pi^{\text{\'et}}(X_s,x)\to\pi^{\text{\'et}}(X,x)\to\pi^{\text{\'et}}(S,s)\to1$$ is exact.
\end{cor}

\begin{proof}
Without loss of generality one may assume $k=\bar{k}$
\cite{Nori}[Part I, Chapter II, Proposition 5]. Now let $(P,G,p)$ be
a $G$-saturated \'etale torsor over $X$, $\pi: P\to X$ be the
structure map $V:=\pi_*O_P$. Let $P\xrightarrow{\phi}
Q\xrightarrow{\varpi} S$ be the stein factorization of the proper
map $P\xrightarrow{\pi}X\xrightarrow{f}S$. Since $f\circ\pi$ is
proper separable $\varpi$ is finite \'etale \cite{EGA}[7.8.10 (i)]. Thus $\phi$ is proper
separable surjective with geometrically connected fibres. But then
the pull back $\phi_s: P_s\to Q_s$ along the rational point
$s\hookrightarrow S$ is also proper separable surjective with
geometrically connected fibres. Hence $O_{Q_s}\to (\phi_s)_*O_{P_s}$
is an isomorphism \cite{EGA}[7.8.6]. But since $\varpi:Q\to S$ is affine,
$H^0(Q_s,O_{Q_s})\cong \varpi_*O_{Q}|_s\cong
f_*\pi_*O_P|_s=f_*V|_s$. Thus base change is satisfied for $V$ at
$s$.

The action $P\times_kG\to P$ induces a map $V\to V\otimes_kk[G]$.
Push it to $S$ we get $f_*V\to f_*V\otimes_kk[G]$. Thus there is an
action of $G$ on $Q=\Spec_{O_S}(f_*V)$ which makes $\phi: P\to Q$ $G$-equivariant. If we pull
back the map $P\to Q\times_SX$ along the rational point $x\in X(k)$,
we get a $G$-equivariant map $t:G\to G'$, where we identify $G$ with
$P\times_Xx$ via the rational point $p\in P$ and
$G':=Q\times_SX\times_Xx=Q_s$ is a $G$-set with a distinguished
point $q:=t(e)$. Since $t$ is $G$-equivariant, the subgroup $H:=t^{-1}(q)$ of $G$ is the
stabilizer of $q$. Now let $h\in H$ be
an element. Consider the $S$-isomorphism $Q\to Q$ induced by $h$.
Evidently $h$ sends $q$ to $q$, and since $Q$ is a connected finite
\'etale cover of $S$, the $S$-isomorphism induced by $h$ must be the
identity\cite{Gr}[Expos\'e I, Corollaire 5.4.]. Hence $H$ acts trivially on $Q$ and in particular it also
acts trivially on $G'$.  But since $t:G\to G'$
is surjective, $G'$ is the quotient of
$G$ by $H$, and $t$ is the quotient map. And also for any $x\in G$, we have
$qxhx^{-1}=t(x)hx^{-1}=t(x)x^{-1}=t(xx^{-1})=q$, so $H$ is a normal subgroup of $G$. Thus $G'$ is the quotient group of $G$ by $H$. The following commutative
diagram:
$$\xymatrix{P\times_kG\ar[r]^{\cong}\ar[d]&P\times_XP\ar[d]\\Q\times_kG'\ar[r]^{\rho}&Q\times_SQ}$$ tells us that $\rho$ is a finite
\'etale surjective $Q$-morphism. Let $r$ be the degree of the connected finite \'etale cover $\varpi:Q\to S$. Then one sees
easily that both $Q\times_kG'$ and $Q\times_SQ$ are finite \'etale
of degree $r$ over $Q$. Thus $\rho$ must be an isomorphsim for it is an isomorphism on all the geometric fibres of $Q$. Because the  Now $\varpi:Q\to S$
has a structure of a $G'$-torsor which satisfies all the conditions
in (3) of our main theorem. So we can use the same argument we have
used in "$(3)\Longrightarrow (1)$" to conclude our proof.
\end{proof}

\section{The Proper Case}
\begin{thm}\footnote{In the original article  \cite{EHV} the statement was that the exactness is equivalent to base change. We add the condition that $f_*V$ is essentially finite to complete the argument.} {\rm (H.Esnault, P.H.Hai , E.Viehweg)} Let $f:X\to S$ be a proper separable morphism with geometrically connected fibres between
two reduced connected proper schemes over a perfect field $k$, $x\in
X(k)$, $s\in S(k)$, $f(x)=s$. Assume further that $S$ is
irreducible. Then the homotopy sequence
$$\pi^N(X_s,x)\rightarrow\pi^N(X,x)\rightarrow \pi^N(S,s)\rightarrow
1$$ is exact if and only if for any $G$-saturated torsor $(P,G,p)\in
N(X,x)$ with structure map $\pi:P\to X$, $V:=\pi_*O_P$ satisfies
base change at $s$ and $f_*V$ is essentially finite.
\end{thm}

\begin{proof} $"\Longleftarrow"$ Since $f_*V$ satisfies base change,
the canonical map $f^*f_*V\to V$ is of the form
$$\Gamma(X_s,V|_{X_s})\otimes_k O_{X_s}\rightarrow V|_{X_s}$$ after
restricting to the fibre $X_s$. Because $f^*f_*V\to V$ is a map of
essentially finite vector bundles, the kernel of it is also a vector
bundle. But the kernel is trivial on $X_s$, so the kernel itself is
trivial. Thus $f^*f_*V\subseteq V$ is a subobject in the category of
essentially finite vector bundles on $X$ and it becomes the maximal
trivial suboject after restricting to $X_s$. Now let $G'$ be the
Tannakian group of the sub Tannakian category of $\Ess(S)$ generated
by $f_*V$. The imbedding $f^*f_*V\to V$ gives us a surjection
$\lambda:G\to G'$. Let  $H$ be the kernel of $\lambda$. Then
$f^*f_*V\to V$ corresponds via Tannakian duality to an inclusion
$M\subseteq k[G]$ in $\Rep_k(G)$. Note that since $M$ comes from an
object in $\Rep_k(G')$ via $\lambda: G\to G'$, so $M\subseteq k[G]$
factors through the inclusion $k[G]^H\subseteq k[G]$. On the other
hand, since we have a surjection $\pi^N(S,s)\to G'$, by
\cite{Nori}[Chapter I, Proposition 3.11] we have a $G'$-saturated
torsor $(P',G',p')\in N(S,s)$ with a map
$$\theta: (P,G,p)\to f^*(P',G',p')$$ in $N(X,x)$ extending
$\lambda$. Let $V':=\pi_*'O_{P'}$, $\pi': P'\to S$. Then since $P\to
f^*P'\cong P/H$ is faithfully flat, $f^*V'\subseteq V$ is a subbundle,
and this subbundle corresponds via Tannakian duality to the
inclusion $k[G]^H\subseteq k[G]$. But clearly $f^*V'\subseteq V$
factors $f^*f_*V\to V$, so $k[G]^H\subseteq k[G]$ factors
$M\subseteq k[G]$, which means $k[G]^H=M$. So we have $V'\cong
f_*V$. Now the triple $(P',G',p')$ satisfies all our conditions in
Theorem 2.1 (3), so we get the exact sequence.

$"\Longrightarrow"$ By Theorem 2.1 we have a $G'$-saturated torsor
$(P', G',p')\in N(S,s)$ and a morphism $$\theta: (P,G,p)\to f^*(P',
G',p')\in N(X,x)$$ such that the induced map ${V}_s'\to (f_*V)_s$ is
an isomorphism, where $V':=\pi_*'O_{P'}$ and $\pi': P'\to S$ is the
structure map. Because $V$ satisfies base change at $s$, there is a
neighborhood $s\in U$ such that $f_*V$ is a vector bundle on $U$ and
the adjunction map $f^*f_*V\to V$ is an imbedding of
subbundles (locally split) on
$f^{-1}(U)$. Since $P\to f^*P'$ is finite faithfully flat, the induced map $f^*V'\to V$ is an imbedding of
subbundles over $X$ (i.e. $V/f^*V'$ is a vector bundle on $X$). Because $f^*V'\to
V$ factors through the adjunction map, we get a map $f^*V'\to f^*f_*V$ which is an imbedding of
subbundles on $f^{-1}(U)$. Since ${V}_s'\cong (f_*V)_s$,
$f^*V'\to f^*f_*V$ is an isomorphism on $f^{-1}(U)$. Hence the
injective map $V'\to f_*V$ is also an isomorphism on $U$. Now by
\cite{EGA}[Th\'eor\`eme 7.7.6] there is a coherent sheaf
$\mathscr{Q}$ on $S$ such that
$$f_*(V/f^*V')\cong\mathscr{H}om_{O_S}(\mathscr{Q},O_S).$$ Since locally
$\mathscr{H}om_{O_S}(\mathscr{Q},O_S)$ is contained in a vector
bundle and by applying the left exact functor $f_*$ to the sequence $0\to f^*V'\to V\to V/f^*V'\to 0$ we have
$$f_*V/V'\subseteq f_*(V/f^*V')=\mathscr{H}om_{O_S}(\mathscr{Q},O_S),$$
so if there is $t\in S\setminus U$ such that $(f_*V/V')_t\neq 0$,
then we can choose an open affine $t\in\Spec(A)\subseteq S$ such
that $$(f_*V/V')|_{\Spec(A)}\subseteq \bigoplus_{i=0}^{n}A_i,$$
where $A_i$ is a rank 1 free $A$-module for all $0\leq i\leq n$.
Notice that since $S$ is integral and $\Spec(A)$ is non-empty, so $A$ is
an integral ring. This implies $f_*V/V'$ is non-zero at the generic
point which contradicts to the fact that $f_*V/V'$ has support in
$S\setminus U$. So $V'\to f_*V$ is an isomorphism on $S$. But $V'$
is certainly essentially finite. This completes the proof.
\end{proof}

\subsection{Application to the K\"unneth formula}

\begin{defn}Let $X$ be a reduced connected locally noetherian scheme over a field $k$
with a rational point $x\in X(k)$. Let $N^F(X,x)$ be the full
subcategory of $N(X,x)$ whose objects consist of pointed torsors
with finite local groups. This category is also filtered so we can
write $\pi^F(X,x):=\varprojlim_{N^F(X,x)}G$. If $X$ is also proper
and $k$ is perfect, then $\pi^F(X,x)$ is the Tannakian group of the
full subcategory of the category of essentially finite vector
bundles $\Ess(X)$ consisting of $F$-trivial bundles, i.e. vector
bunldes which are trivial after pull back along some relative
Frobenius $\phi_{(-t)}: X^{(-t)}\to X$ with $t\in\N$.\end{defn}

\begin{cor}Let $X$ and $Y$ be two reduced connected proper schemes over a perfect field $k$. Let $x\in X(k)$,
$y\in Y(k)$. Then the canonical map
$$\pi^F(X\times_kY,(x,y))\to\pi^F(X,x)\times_k\pi^F(Y,y)$$ is an
isomorphism of $k$-group schemes.
\end{cor}

\begin{proof} As usual we may assume $k=\bar{k}$. We will use the obvious analogues of Theorem 3.1
to prove this theorem. Note that after replacing $\pi^N(X,x)$ by
$\pi^F(X,x)$, "torsor" by "local torsor" (torsors whose groups are
local), essentially finite vector bundle by $F$-trivial vector
bundle, Theorem 2.1 and Theorem 3.1 are still true.

To prove this corollary we only need to show that the sequence
$$1\to\pi^F(Y,y)\to \pi^F(X\times_kY,(x,y))\to \pi^F(X,x)\to 1$$ is
exact. So we have to check that for any $G$-saturated local torsor
$(P,G,p)\in N^F(X,x)$, $V:=\pi_*O_P$ ($\pi: P\to X$ is the structure
map) satisfies base change at $x$ and $f_*V$ is an $F$-trivial
vector bundle.

Now suppose that $V$ is trivialized by $$X^{(-t)}\times_k
Y^{(-t)}=(X\times_kY)^{(-t)}\to X\times_kY.$$ Consider the following
commutative diagram
$$\xymatrix{Y\ar[r]^-{x}\ar@{=}[d]&X^{(-t)}\times_kY\ar[d]^-{\phi_{(-t)}\times id}\ar[r]^-{p_1}&X^{(-t)}\ar[d]^-{\phi_{(-t)}}\\Y\ar[r]^-x&X\times_kY\ar[r]^-{f}&X}.$$ Let
$W$ be the pull back of $V$ via $X^{(-t)}\times_kY\to X\times_kY$.
 Let
$W$ be the pull back of $V$ via $X^{(-t)}\times_kY\to X\times_kY$.
Since $W$ has trivial fibres along the projection $p_2:
X^{(-t)}\times_kY\to Y$ and $X^{(-t)}$ is proper separable and
geometrically connected scheme, if we set 
$E:={p_2}_*W$ then by \cite{Mum}[Chapter 2, \S 5, Corollary 2] we see that $E$ is a vector bundle and that the canonical map $p_2^*E\to W$ is an isomorphism. Hence for all closed point of $X$
(or equivalently $X^{(-t)}$), $V|_Y\cong W|_Y\cong E$ i.e. $V$ has
constant fibres along $f: X\times_kY\to X$. Consequently base change
is satisfied for $V$ along $f$ (at any point of $X$). On the other
hand we have the following trivial cartesian diagram
$$\xymatrix{X^{(-t)}\times_kY\ar[r]^-{p_2}\ar[d]^-{p_1}&Y\ar[d]^-b\\X^{(-t)}\ar[r]^-a&\Spec(k)},$$ where $a$ and $b$ are structure maps.
Because of base change we have $a^*b_*E\cong {p_1}_*{p_2}^*E$. This
implies ${p_1}_*W={p_1}_*{p_2}^*E$ is a trivial vector bundle. But
since base change is satisfied for $V$ along $f$, so we have a
canonical isomorphism
$${p_1}_*W={p_1}_*(\phi_{(-t)}\times id)^*V\cong\phi_{(-t)}^*f_*V.$$
Thus $\phi_{(-t)}^*f_*V$ is a trivial vector bundle. By definition
$f_*V$ is $F$-trivial.
\end{proof}

\begin{rmks}
(1) Here we didn't assume $X$ or $Y$ is irreducible, this is because
we have only used the sufficiency part of Theorem 3.1 in which only
the citation of Theorem 2.1 used the irreduciblity. But we only used
$(3)\Rightarrow (1)$ part of Theorem 2.1 where the
irreduciblity plays no role.\\
(2) This corollary gives another way to see \cite{MS}[Proposition
2.1] which is the key point in the proof of the K\"unneth formula
for Nori's fundamental group. But unfortunately, for the full proof
of of K\"unneth formula we have to use the same trick employed in
\cite{MS} to reduce the problem for $\pi^N$ to the problem for
$\pi^F$. At the moment, I can not find any easy way to reduce the
problem to $\pi^F$ using our language here.
\end{rmks}

\section{A counterexample}

\begin{lem}Let $X$ and $Y$ be two reduced connected schemes locally of finite type over an algebraically closed field $k$. Let $x\in X(k)$,
$y\in Y(k)$. If the canonical map
$$\pi^N(X\times_kY,(x,y))\to\pi^N(X,x)\times_k\pi^N(Y,y)$$ is an
isomorphism of $k$-group schemes, then for any $(P,G,p)\in N(X\times_k Y,(x,y))$ and any other point $x'\in X(k)$, the restriction of
$(P,G,p)$ to $Y$ along $x$ and $x'$ are isomorphic.\end{lem}

\begin{proof}By the assumption we have a $(P_1,G_1,p_1)\in N(X,x)$ and $(P_2,G_2,p_2)\in N(Y,y)$ and a morphism $$(P_1\times_k P_2, G_1\times_kG_2,p_1\times p_2)\to (P,G,p)$$in $N(X\times_kY,(x,y))$. But the restrictions of $(P_1\times_k P_2, G_1\times_kG_2,p_1\times p_2)$ to $Y$ along $x$ and $x'$ are all isomorphic to $(G_1\times_kX\times_kP_2,G_1\times_kG_2,e\times p_2)$. This implies the restrictions of $(P,G,p)$ are all isomorphic to $$((G_1\times_kX\times_kP_2)\times^{(G_1\times_kG_2)}G,G,p)$$ the contracted product. This finishes the proof.
\end{proof}

Now consider $k$ an algebraically closed field of characteristic 2,
$X=\A^1_k$, $Y$ is a supersingular elliptic curve over $k$. In the
following we will construct an $\alpha_2$-torsor $\pi': Q\to
X\times_kY$ in FPQC-topology, and we will show that the restrictions
of $Q$ to $Y$ along the two rational points $x=0$ and $x=1$ in $X$
can not be isomorphic. This shows by our lemma that the canonical
map
$$\pi^N(X\times_kY,(x,y))\to\pi^N(X,x)\times_k\pi^N(Y,y)$$ could not
be an isomorphism for any rational point $y\in Y$ and $x=0$ or
$x=1$.

\begin{construction}Suppose $\pi:P\to Y$ be a non-trivial $\alpha_2$-torsor in
FPQC-topology. Note that we can choose $\pi:P\to Y$ to be the
Frobenius endomorphism $F:Y\to Y$, it naturally carries a
translation by $Y[F]\cong \alpha_2$ which makes it into a
non-trivial $\alpha_2-$torsor.

Let $V:=\pi_*O_P$. Let $\mathcal {L}$ be the cokernel of the
structure map $O_Y\to V$, then $\mathcal{L}$ is an essentially
finite line bundle, so it has degree 0. Since
$H^1(Y,\mathcal{L}^{-1})=\Ext^1(O_Y,\mathcal{L}^{-1})\neq 0$, so by
Riemann-Roch $h^0(Y,\mathcal{L}^{-1})=h^1(Y,\mathcal{L}^{-1})\neq
0$, but this implies $\mathcal{L}^{-1}$ is $O_Y$, hence we have
$\mathcal{L}\cong O_Y$. This gives us an exact sequence
$$0\to O_Y\to V\to O_Y\to 0.$$

We know that $P\to Y$ already becomes a trivial
torsor after pulling back along the relative Frobenius
$\phi_{(-1)}:Y^{(-1)}\to Y$. Thus after choosing a section
$Y^{(-1)}\to P\times_YY^{(-1)}$ we get an $Y^{(-1)}$-scheme
isomorphism $\alpha_2\times_kY^{(-1)}\cong P\times_YY^{(-1)}$ which
gives us an isomorphism of $O_{Y^{(-1)}}$-algebras
$$\delta:\xymatrix{{\phi_{(-1)}}^*V\ar[rr]^-{\cong}&& O_{Y^{(-1)}}[T]/T^2}$$ making the
diagram $(*)$:
$$\xymatrix{0\ar[r]&O_{Y^{(-1)}}\ar[r]\ar@{=}[d]&{\phi_{(-1)}}^*V\ar[r]\ar[d]^{\delta}_{\cong}&O_{Y^{(-1)}}\ar@{=}[d]\ar[r]&0\\0\ar[r]&O_{Y^{(-1)}}\ar[r]&O_{Y^{(-1)}}[T]/T^2\ar[r]&O_{Y^{(-1)}}\ar[r]&0}$$commutative.
By Grothendieck's FPQC descent theory there is an essentially unique
isomorphism of $O_{Y^{(-1)}\times_YY^{(-1)}}$-algebras $\varepsilon$
corresponding to $V$
$$\xymatrix{p_1^*{\phi_{(-1)}}^*V\ar[d]^{p_1^*\delta}\ar[r]^{\cong}&p_2^*{\phi_{(-1)}}^*V\ar[d]^{p_2^*\delta}
\\O_{Y^{(-1)}\times_YY^{(-1)}}[T]/T^2\ar[r]^{\varepsilon}&O_{Y^{(-1)}\times_YY^{(-1)}}[T]/T^2},$$
where $p_1$ and $p_2$ are the two projections of
$Y^{(-1)}\times_YY^{(-1)}$. From the commutative diagram $(*)$, we
know that this $\varepsilon$ is expressible by a matrix
$$\left(\begin{array}{ccc} 1 & a \\
0& 1
\end{array}\right)$$ in
$GL_2(\Gamma(Y^{(-1)}\times_YY^{(-1)},O_{Y^{(-1)}\times_YY^{(-1)}}))$
regarding $\{1,T\}$ as a basis for
$O_{Y^{(-1)}\times_YY^{(-1)}}[T]/T^2$. Since $P$ is not a trivial
torsor $a\neq 0$. Moreover since  $T\mapsto a+T$ by $\varepsilon$ we
have $(T+a)^2=0$, thus $a^2=T^2+a^2=(T+a)^2=0$.

Let $x$ be the indeterminate in $X=\A^1_k=\Spec(k[x])$. Let $A:=X\times_kY^{(-1)}\times_YY^{(-1)}$,  $\mathscr{A}:=O_{X\times_kY^{(-1)}\times_YY^{(-1)}}$ and $T$ be an indeterminate. Then the
$2\times2$-matrix: $$\left(\begin{array}{ccc} 1 & ax \\
0& 1
\end{array}\right)$$ in $GL_2(A,O_{A})$ determines an isomorphism
$$\varepsilon':\xymatrix{\mathscr{A}[T]/T^2\ar[rr]^{\cong}&&\mathscr{A}[T]/T^2}$$$$\ \ \ \ \ \ \ \ \ \xymatrix{T\ar@{|->}[rr]&&ax+T}.$$
  One checks readily that the pair
$$(\mathscr{A}[T]/T^2,\varepsilon')$$ gives us a descent data of affine schemes, i.e. the cocycle condition is
satisfied and the isomorphism $\varepsilon'$ is an automorphism of
the $\mathscr{A}$-algebra $\mathscr{A}[T]/T^2$. So it gives us an
affine scheme $\pi':Q\to X\times_kY$, which is obviously finite and
faithfully flat.

Next we will check that $Q$ is an $\alpha_2$-torsor in FPQC
topology. Let $S:=T$. Now we consider the commutative diagram of
$\mathscr{A}$-algebras
$$\xymatrix{\mathscr{A}[T]/T^2\otimes_{\mathscr{A}}\mathscr{A}[S]/S^2\ar[rr]^{\varepsilon'\otimes \varepsilon'}\ar[d]^{\lambda}&&\mathscr{A}[T]/T^2\otimes_{\mathscr{A}}\mathscr{A}[S]/S^2\ar[d]^{\lambda}\\\mathscr{A}[T]/T^2\otimes_{\mathscr{A}}\mathscr{A}[S]/S^2
\ar[rr]^{\varepsilon'\otimes id}&&\mathscr{A}[T]/T^2\otimes_{\mathscr{A}}\mathscr{A}[S]/S^2},$$
where $\lambda$ denotes the sheaf version of the canonical map
$$\alpha_{2,A}\times_{A}\alpha_{2,A}\xrightarrow{\cong}\alpha_{2,A}\times_{A}\alpha_{2,A}$$ sending $(x,y)\mapsto (x,xy)$.
This commutative diagram means that $\lambda$ defines an isomorphism between the descent data of $Q\times_{k}\alpha_{2,k}$ and $Q\times_{X\times_kY}Q$.
 Let $$\rho:Q\times_{k}\alpha_{2,k}\to Q\times_{X\times_kY}Q$$ be the corresponding isomorphism.
 It is clear that the composition of $\rho$ with the second projection of $Q\times_{X\times_kY}Q$
 defines an action$$\rho:Q\times_{k}\alpha_{2,k}\longrightarrow Q$$$$(q,g)\mapsto q\cdot g$$ of $\alpha_2$ on $Q$,
 so $\rho$ is the canonical map $(q,g)\mapsto (q,q\cdot g)$. This shows that $\pi':Q\to X\times_kY$ is an $\alpha_2$-torsor in FPQC-topology.

It is clear that the fiber of $Q$ on $Y$ along $x=0$ is the trivial
torsor and the fiber along $x=1$ is just $\pi:P\to X$ which is
non-trivial. This violates the necessary condition in our Lemma 4.1.
\end{construction}

\begin{rmk} Let $Y_i\ (i=0,1)$ be the fibre of $x=i$ along the map
$X\times_kY\to X$ and let $W:={\pi'}_*Q$. In the above construction
we see that $W|_{Y_0}=O_{Y_0}\oplus O_{Y_0}$ and $W|_{Y_1}=V$. If
the chosen non-trivial torsor $\pi: P\to Y$ is the relative
Frobenius endomorphism $F: Y\to Y$, then we have
$\dim_k\Gamma(Y,W|_{Y_1})=1$. But  $\dim_k \Gamma(Y,W|_{Y_0})=2$.
This tells us that $W$ does not satisfy base change at $x=0$. So it
provides an example for which the base change condition in Theorem
2.1 is not satisfied. But note that the normality condition is still
OK, since in this example the group is commutative.
\end{rmk}

\begin{rmk} The above construction actually shows that for any smooth proper connected scheme
$Y$ over an algebraically closed field $k$ of characteristic 2 which
admits a non-trivial $\alpha_2$-torsor, and for any rational points
$(x,y)\in \A_k^1\times_kY$ the K\"unneth formula does not hold for
$\A_k^1\times_kY$. In fact the only place where we used the
assumption that $E$ is an elliptic curve is in the argument to show
that $\mathcal{L}$ is $O_Y$, but everything still works without
knowing that $\mathcal{L}\cong O_Y$, it's just that the computation
is a little more complicated. Furthermore one can show easily that
in this situation if K\"unneth formula holds for one rational point
$(x,y)\in \A_k^1\times_kY(k)$ then it holds for any other rational
point.
\end{rmk}

\begin{rmk} H\'el\`ene Esnault and Andre Chatzistamatiou pointed to us the following non-constructive improvement of the above example.  Thanks to their suggestion our
counterexample may work for any algebraically closed field $k$ of
characteristic $p>0$. Now we consider the exact sequence of abelian
sheaves in the FPPF-topology
 $$0\to \alpha_p\to \G_a\xrightarrow{F}\G_a\to 0, $$ we then get a long exact sequence of abelian groups: $$\cdots\to H_{fl}^0(Y,\G_a)
 \to H_{fl}^1(Y,\alpha_p)\to H_{fl}^1(Y,\G_a)\xrightarrow{H^1(F)} H_{fl}^1(Y,\G_a)\to \cdots$$
 where $Y$ is a smooth proper connected $k$-variety so that the map $H^1(F)$ has none trivial kernel (e.g. a supersingular elliptic curve). Then we can choose some $a\neq 0$ in
 the that kernel. We can also put $\A_k^1\times_kY$ into the above long exact sequence instead of $Y$,
then since $$a\otimes x\in
H_{fl}^1(\A_k^1\times_kY,\G_a)=H_{fl}^1(Y,\G_a)\otimes_kk[x]$$ is
still in the
 kernel of $H^1(F)$, so
 we can choose an element $b\in H_{fl}^1(\A_k^1\times_kE,\alpha_p)$ such that $b\mapsto a\otimes x$. But then $b$ is an $\alpha_p$-torsor with non-constant fibres
 along the projection $\A_k^1\times_kY\to \A_k^1$ because $b$ has trivial image at $x=0$ and non-trivial image at $x=1$. This can not happen if K\"unneth formula was true.
\end{rmk}


\begin{thebibliography}{9999999}
\bibitem[EGA]{EGA} Alexender Grothendieck, EGA III: \'Etude
Cohomologique des Faisceaux Coh\'erents, 2-i\`eme partie, Publ.
Math. I.H.E.S. 17 (1963), 5-91.
\bibitem[EHV]{EHV} Esnault, H., Hai P.-H., Viehweg, E., On the
homotopy exact sequence for Nori's fundamental group, arxiv.
\bibitem[EPS]{EPS} H\'el\`ene Esnault, Ph\`ung H\^{o} Hai, Xiaotao
Sun, On Nori's Fundamental Group Scheme, Progress in Mathematics,
Vol.265, 366-398, Birkh\"{a}user Verlag Basel/Switzerland, 2007;
\bibitem[Gr]{Gr} Alexander Grothendieck, Rev\^{e}tements \'Etales
et Groupe Fondamental, SGA 1, Springer-Verlag, 1971.
\bibitem[Hart]{Hart} R.Hartshorne, Algebraic Geometry, Springer-Verlag,1997;
\bibitem[MS]{MS}Mehta, V.B., Subramanian, S.: On the fundamental
group scheme, Invent. math. 148 (2002), 143-150.
\bibitem[Mum]{Mum} Mumford, D,  Abelian Varieties, Tata Institute of Fundamental Research, Mumbai, Corrected Reprint 2008.
\bibitem[Nori]{Nori} Nori, M., The fundamental group schmes,
Proc.Indian Aacd.Sci. 91(1982), 73-122.
\end{thebibliography}
\end{document}